\documentclass[12pt,leqno]{article}
\usepackage{amsmath,amssymb,amscd,latexsym,bm}

\usepackage[all]{xy}
\newcommand{\A}{{\mathbb{A}}}
\newcommand{\C}{{\mathbb{C}}}

\newcommand{\F}{{\mathbb{F}}}
\newcommand{\oF}{\overline{\F}}
\newcommand{\Ge}{\mathbb{G}}

\newcommand{\Q}{{\mathbb{Q}}}
\newcommand{\oQ}{\overline{\Q}}
\newcommand{\R}{{\mathbb{R}}}

\newcommand{\Z}{{\mathbb{Z}}}

\newcommand{\ddet}{\mathrm{det}}

\newcommand{\id}{\mathrm{id}}
\newcommand{\Li}{\mathrm{Li}}
\newcommand{\Map}{\mathrm{Map}}

\newcommand{\Mor}{\mathrm{Mor}}

\newcommand{\ord}{\mathrm{ord}}

\newcommand{\rk}{\mathrm{rk}\,}

\newcommand{\spec}{\mathrm{spec}\,}

\newcommand{\tr}{\mathrm{tr}}
\newcommand{\Aut}{\mathrm{Aut}}

\newcommand{\End}{\mathrm{End}}

\newcommand{\Ext}{\mathrm{Ext}}

\newcommand{\Gal}{\mathrm{Gal}}
\newcommand{\GL}{\mathrm{GL}}

\newcommand{\Imm}{\mathrm{Im}\,}

\newcommand{\ind}{\mathrm{ind}}

\newcommand{\Ker}{\mathrm{Ker}\,}

\newcommand{\Ob}{\mathrm{Ob}\,}
\newcommand{\Perv}{\mathrm{Perv}}
\newcommand{\perv}{\mathrm{perv}}

\newcommand{\RRe}{\mathrm{Re}}

\newcommand{\nVec}{\mathrm{Vec}}
\newcommand{\Ch}{{\mathcal C}}

\newcommand{\Dh}{\mathcal{D}}

\newcommand{\Mh}{{\mathcal M}}
\newcommand{\Nh}{{\mathcal N}}

\newcommand{\Vh}{\mathcal{V}}

\newcommand{\Yh}{\mathcal{Y}}

\newcommand{\eo}{\mathfrak{o}}

\newcommand{\eX}{{\mathfrak X}}

\newcommand{\ceX}{\check{\eX}}

\newcommand{\orho}{\overline{\rho}}

\newcommand{\tvarphi}{\tilde{\varphi}}

\newcommand{\oalpha}{\overline{\alpha}}
\newcommand{\ooalpha}{\overline{\oalpha}}
\newcommand{\tV}{\tilde{V}}

\newcommand{\hzeta}{\hat{\zeta}}

\newcommand{\silo}{\xrightarrow{\,\sim\,}}

\newcommand{\colim}{\mathrm{colim}}

\newcommand{\Ind}{\mathrm{Ind}}

\newcommand{\halb}{\frac{1}{2}}
\newtheorem{theorem}{Theorem}[section]

\newtheorem{prop}[theorem]{Proposition}

\newtheorem{punkt}[theorem]{$\!\!$}
\newtheorem{prediction}[theorem]{Prediction}
\newcounter{aufz}
\newenvironment{proofof}{\noindent {\bf Proof of}}{\mbox{}\hspace*{\fill}$\Box$}
\newenvironment{proof}{\noindent {\bf Proof}}{\mbox{}\hspace*{\fill}$\Box$}

\newcommand{\verk}{\mbox{\scriptsize $\,\circ\,$}}
\makeatletter
\def\moverlay{\mathpalette\mov@rlay}
\def\mov@rlay#1#2{\leavevmode\vtop{%
   \baselineskip\z@skip \lineskiplimit-\maxdimen
   \ialign{\hfil$\m@th#1##$\hfil\cr#2\crcr}}}
\newcommand{\charfusion}[3][\mathord]{
    #1{\ifx#1\mathop\vphantom{#2}\fi
        \mathpalette\mov@rlay{#2\cr#3}
      }
    \ifx#1\mathop\expandafter\displaylimits\fi}
\makeatother

\newcommand{\cupdot}{\charfusion[\mathbin]{\cup}{\cdot}}

\begin{document}
\title{Is there a Birch and Swinnerton-Dyer conjecture for Dedekind zeta functions?}
\author{Christopher Deninger\footnote{Funded by the Deutsche Forschungsgemeinschaft (DFG, German Research Foundation) under Germany's Excellence Strategy EXC 2044--390685587, Mathematics M\"unster: Dynamics--Geometry--Structure and the CRC 1442 Geometry: Deformations and Rigidity}}
\date{}
\maketitle
\centerline{\textit{Dedicated to the memory of Tobias Kreutz}}


\section{Introduction}

For an elliptic curve $E / \Q$, the Birch and Swinnerton-Dyer conjecture asserts that
\[
\rk E (\Q) = \ord_{s = 1} L (E , s) \; .
\]
There is also a prediction for the leading coefficient $L^* (E ,1)$ of the Taylor series at $s = 1$. This conjecture has inspired a huge body of work. The point $s = 1$ is the ``central point'' for the functional equation of $L (E,s)$ under the substitution $s \mapsto 2-s$. For the Dedekind zeta function $\zeta_K (s)$ of a number field $K$ the functional equation relates the values at $s$ and $1-s$ and the central point is $s = 1/2$. To this day, there is no suggestion of a group or vector space $V_K$ attached to $K$ in a natural way for which we would at least conjecturally have
\begin{equation}
\label{eq:1n}
\dim V_K = \ord_{s = 1/2} \zeta_K (s) \; .
\end{equation}
Also, there is no prediction for $\zeta^*_K (1/2)$ in the spirit of the BSD-conjecture. In the function field case the corresponding problem has been solved in the beautiful paper \cite{R}. For a non-abelian number field $K$ of degree $2^5 3$ constructed by Serre in \cite{Se}\,\S\,9, Armitage showed in \cite{A} that $\zeta_K (1/2) = 0$. In section 3 of \cite{M} a proof of Kurokawa's result is given that there exist Galois extensions $K_{\nu} / \Q$ of degree $2^{3 \nu}$ with $\ord_{s = 1/2} \zeta_{K_{\nu}} (s) \ge 2 \nu$. The idea of this proof is different from the one in \cite{A} and uses results by Brauer and Fr\"ohlich. For Galois extensions $K / \Q$ with Galois group $G$, the Dedekind zeta function factors as a product of Artin $L$-functions where $\pi$ runs over the isomorphism classes of irreducible complex representations of $G$,
\[
\zeta_K (s) = \prod_{\pi} L (\pi , s)^{\deg \pi} \; .
\]
Serre conjectured that $L (\pi , s)$ vanishes at $s = 1/2$ only if the functional equation forces it to vanish, and in this case the vanishing order should be one. More explicitely, if $\pi$ is not selfdual or if the root number $W (\pi)$ is $+1$ as for orthogonal $\pi$, \cite{FQ}, then the functional equation does not imply vanishing and Serre expects that $L (\pi , 1/2) \neq 0$. On the other hand, for symplectic $\pi$ with $W (\pi) = -1$ the functional equation implies that $L (\pi , s)$ vanishes to odd order and the conjecture says that $\ord_{s = 1/2} L (\pi , s) = 1$. In particular it would follow that $\zeta_K (1/2) \neq 0$ for abelian extensions $K / \Q$. 

 The reference \cite{O} contains numerical evidence for Serre's conjecture and may serve as a survey for this problem. 

In earlier work, we proposed a conjectural global cohomological formalism for arithmetic schemes which would exhibit the (completed) Dedekind zeta function as an alternating product of zeta-regularized ``characteristic polynomials'', see \ref{t2.4} below. The formalism predicts the existence of a functor $K \mapsto V_K$ from the category of number fields into symplectic complex vector spaces with a $*$-operator whose dimensions would equal the vanishing orders of Dedekind zeta functions at $s = 1/2$. Together with a conjecture of Serre on the vanishing order of $\zeta_K (s)$ at $s = 1/2$ we obtain the precise Prediction \ref{t2.14} below about the extra properties that a natural functor $K \mapsto V_K$ satisfying formula \eqref{eq:1n} should have. A less precise prediction was already given in \cite{D2} to which the present note is a sequel. \\
Theorem \ref{t2.15} states that abstractly functors as in Prediction \ref{t2.14} exist and Theorem \ref{t2.17} asserts that they are all isomorphic and determines their common automorphism group. The very formal proofs are given in section \ref{sec:3}. The problem remains to find a \textit{natural} candidate for the functor $V \mapsto V_K$. We hope that the precise conditions listed in Prediction \ref{t2.14} will help to discover a natural construction of $V_K$. In the last section we point out some problems with trying to use extension groups of exponential motives for this purpose. 

I would like to thank Javier Fres\'an and Peter Jossen very much for allowing me to reproduce the argument for their vanishing result, Theorem \ref{t4.1}. I am also very grateful to Javier Fres\'an and the referee for their careful reading of this note and their many detailed suggestions to improve the exposition. 
\section{Dedekind zeta functions at $s = 1 / 2$ and cohomology} \label{sec:2}

In this section we mostly recall and discuss some material from \cite{D2} in a slightly different fashion and with some extra arguments. The conclusions up to Prediction \ref{t2.14} are speculative because they depend on properties of cohomology theories which have not yet been shown to exist. At the end of the section we state two theorems that are motivated by the discussion.

For a finite extension $K / \Q$ let $H^i (\Yh_K , \Ch)$ be the conjectural complex cohomology of the arithmetic compactification $\Yh_K = \overline{\spec} \eo_K$ of $\spec \eo_K$ considered in \cite{D2} \S\,2.  It comes with an operator $\theta$ which behaves as a derivation with respect to cup product. A geometric analogue of this cohomology theory exists for a certain class of foliated $\R$-dynamical systems. There $\Ch$ denotes the sheaf of germs of smooth complex valued functions which are locally constant on the leaves, \cite{D3}. In the dynamical context, $\theta$ is the infinitesimal generator of the induced $\R$-action on cohomology. We expect the following:
\begin{punkt}\label{t2.1}\em
We have $H^i (\Yh_K , \Ch) = 0$ for $i \ge 3$. Moreover $H^0 (\Yh_K , \Ch) = \C$ with $\theta=0$ and there is a canonical $\theta$-equivariant trace isomorphism $\tr : H^2 (\Yh_K , \Ch) \to \C (-1)$ where $\C (-1) = \C$ as a vector space equipped with the endomorphism $\theta = \id$. As suggested by Baptiste Morin, there should be a cycle class map from the Arakelov Chow group of $\Yh_K$ to $H^2 (\Yh_K , \Ch)$ such that the following diagram commutes
\begin{equation} 
\label{eq:1}
\vcenter{\xymatrix{
CH^1 (\Yh_K) \ar[d]_{cl} \ar[rr]^-{\deg} & & \R \ar@{^{(}->}[d] \\
(H^2 (\Yh_K , \Ch))^{\theta = 1} \ar@{=}[r] & H^2 (\Yh_K , \Ch) \ar[r]^-{\overset{\tr}{\sim}} & \C \; .
}}
\end{equation}
In particular, $H^2 (\Yh_K , \Ch)$ has a canonical $\theta$-invariant $\R$-structure.
\end{punkt}
\begin{punkt} \label{t2.2} \em
A cup product pairing
\[
H^1 (\Yh_K , \Ch) \times H^1 (\Yh_K , \Ch) \xrightarrow{\cup} H^2 (\Yh_K , \Ch) \overset{\tr}{\silo} \C (-1) \; .
\]
For $h_1 , h_2 \in H^1 (\Yh_K , \Ch)$ it would follow that
\[
h_1 \cup h_2 = \theta (h_1 \cup h_2) = \theta h_1 \cup h_2 + h_1 \cup \theta h_2 \; .
\]
\end{punkt}
\begin{punkt} \label{t2.3} \em 
An antilinear operator $* : H^1 (\Yh_K , \Ch) \silo H^1 (\Yh_K , \Ch)$ with $*^2 = -1$ and $\theta \verk \ast = \ast \verk \theta$ and such that $\langle h , h' \rangle = \tr (h \cup *h')$ defines a hermitian scalar product on $H^1 (\Yh_K , \Ch)$. It follows that $\theta - \halb$ is skew symmetric and in particular semisimple on $H^1 (\Yh_K , \Ch)$.
\end{punkt}
\begin{punkt} \label{t2.4} \em 
The relation to the (completed) Dedekind zeta function of $K$ is given by the following formula, \cite{D0}
\[
\hzeta_K (s) = \prod^2_{i=0} \ddet_{\infty} \Big( \frac{1}{2\pi} (s- \theta) \mid H^i (\Yh_K , \Ch) \Big)^{(-1)^{i+1}} \; .
\]
By \ref{t2.1} the first order poles at $s = 0,1$ of $\hzeta_K (s)$ would be accounted for by the eigenvalues $0$ and $1$ of $\theta$ on $H^0 (\Yh_K , \Ch)$ and on $H^2 (\Yh_K , \Ch)$. For all $\rho \in \C \setminus \{ 0, 1 \}$, by the semisimplicity of $\theta$ we would have
\begin{equation} \label{eq:2n}
\ord_{s = \rho} \hzeta_K (s) = \dim_{\C} (H^1 (\Yh_K , \Ch))^{\theta = \rho} \; .
\end{equation}
Here $(\;)^{\theta = \rho}$ denotes the $\rho$-eigenspace of $\theta$. Since $\hzeta_K (s)$ is known to have infinitely many zeros, $H^1 (\Yh_K , \Ch)$ must be infinite dimensional. 
\end{punkt}
\begin{punkt} \label{t2.5} \em 
I expect the conjectural cohomology theory $H^i (\eX, \Ch)$ for (compactified) arithmetic schemes $\eX$ over $\spec \Z$ (or $\overline{\spec} \Z$) to take values in the following category $K_{\R}$ which is a refinement of the category of $\C$-vector spaces. $K_{\R}$ is a $\Z / 2$-graded version of Kottwitz's category $Kt_{\R}$ mentioned in \cite[Conjecture 9.5]{S} in the context of Weil cohomology theories for varieties over $\oF_p$. Objects of $K_{\R}$ are $\Z / 2$-graded complex vector spaces $V = V^0 \oplus V^1$ with an antilinear isomorphism $\tau$ respecting the grading and such that $\tau^2 = (-1)^{\nu}$ on $V^{\nu}$. Thus $V^0$ carries a real structure and $V^1$ carries a quaternionic structure. If $\tau (v \cup w) = \tau (v) \cup \tau (w)$ for $v \in H^i (\eX , \Ch)$ and $w \in H^j (\eX , \Ch)$, then for the homogenous parts we have
\[
H^i (\eX , \Ch)^{\nu} \cup H^j (\eX , \Ch)^{\mu} \subset H^{i+j} (\eX , \Ch)^{\nu + \mu} \; .
\]
In our case $\eX = \Yh_K$, the spaces $H^0 (\Yh_K , \Ch)$ and $H^2 (\Yh_K , \Ch)$ must have $\Z / 2$-grading zero since they are $1$-dimensional. This means that they have a canonical real structure. This is compatible with diagram \eqref{eq:1}. On $H^1 (\Yh_K , \Ch)$ we expect that $\tau = *$ and that $H^1 (\Yh_K , \Ch)$ therefore has $\Z / 2$-grading $1$. We now explain how this implies the property $\overline{\langle h , h' \rangle} = \langle h' , h \rangle$ of the expected hermitian scalar product on $H^1 (\Yh_K , \Ch)$ in \ref{t2.3}. By \eqref{eq:1} the real structure $\tau$ on $H^2 (\Yh_K , \Ch)$ is compatible with $\tr$, that is, complex conjugation on $\C$ corresponds to $\tau$ on $H^2 (\Yh_K , \Ch)$. Since $\tau = *$ on $H^1 (\Yh_K , \Ch)$ we get
\begin{align*}
\overline{\langle h, h' \rangle} & = \overline{\tr (h \cup *h')} = \tr \, \tau (h \cup * h') = \tr (\tau (h) \cup \tau (* h')) \\
& = \tr (*h \cup ** h') = - \tr (* h \cup h') = \tr (h' \cup *h) = \langle h' , h \rangle \; .
\end{align*}
\end{punkt}
The properties \ref{t2.2}, \ref{t2.3} and \ref{t2.4} imply that:
\begin{punkt} \label{t2.6} \em 
The alternating pairing 
\[
\cup_{\tr} = \tr \verk \cup : H^1 (\Yh_K , \Ch) \times H^1 (\Yh_K , \Ch) \to \C (-1)
\]
from \ref{t2.2} induces perfect pairings between the finite-dimensional eigenspaces \\
$(H^1 (\Yh_K , \Ch))^{\theta = \rho}$ and $(H^1 (\Yh_K , \Ch))^{\theta = 1-\rho}$ for all $\rho \in \C$ in accordance with the functional equation of $\hzeta_K (s)$. Since $*$ is antilinear and commutes with $\theta$, it sends $(H^1 (\Yh_K , \Ch))^{\theta = \rho}$ to $(H^1 (\Yh_K , \Ch))^{\theta = \overline{\rho}}$. Note that for eigenvalues $\rho$ of $\theta$, since $\theta - \halb$ is skew symmetric, we have $\orho - \halb = - (\rho - \halb)$ hence $\orho = 1 - \rho$, and therefore $\RRe \rho = \halb$ in accordance with the Riemann hypothesis. 
\end{punkt}
\begin{punkt} \label{t2.7} \em
For any homomorphism $\alpha : K \hookrightarrow L$ of number fields, let 
\[
f = \overline{\spec} \alpha : \Yh_L \to \Yh_K
\]
be the induced map. It gives a contravariant map $f^*$ between the Arakelov Chow groups making the following diagram commutative
\begin{equation}
\label{eq:2}
\vcenter{\xymatrix{
CH^1 (\Yh_K) \ar[d]_{f^{\ast}} \ar[r]^-{\deg} & \R \ar[d]^{[L : K]} \\
CH^1 (\Yh_L)\ar[r]^-{\deg} & \R  \; .
}}
\end{equation}
The induced homomorphism $f^*$ of algebras $H^{\bullet} (\Yh_K, \Ch)$ should commute with $\theta, *$ and with $cl$. In particular, diagram \eqref{eq:1} then gives the commutative diagram
\begin{equation}
\label{eq:3}
\vcenter{\xymatrix{
H^2 (\Yh_K , \Ch) \ar[r]^-{\tr} \ar[d]_{f^*} & \C \ar[d]^{[L : K]} \\
H^2 (\Yh_L,\Ch) \ar[r]^-{\tr} & \C \; .
}}
\end{equation}
Diagram \eqref{eq:3} implies that for every automorphism $\sigma$ of $K$ the induced action by $(\overline{\spec} \sigma)^*$ on $H^2 (\Yh_K , \Ch)$ is trivial. It follows that $\Aut (K)$ respects both the alternating pairing $\cup_{\tr} : H^1 (\Yh_K , \Ch) \times H^1 (\Yh_K , \Ch) \to \C (-1)$ in \ref{t2.6} and the scalar product $\langle , \rangle$ on $H^1 (\Yh_K , \Ch)$ in \ref{t2.3}. 
\end{punkt}
\begin{punkt} \label{t2.8} \em 
We are not specific about the precise nature of $H^1 (\Yh_K , \Ch)$ as a topological vector space. In any case we expect the direct sum of their $\theta$-eigenspaces to be dense in $H^1 (\Yh_K , \Ch)$. If we replace the cohomology groups $H^1 (\Yh_K , \Ch)$ by the direct sum of their $\theta$-eigenspaces, namely write $H^1 (\Yh_K , \Ch)$ for this direct sum, we obtain a unique map $f_* : H^1 (\Yh_L , \Ch) \to H^1 (\Yh_K , \Ch)$ dual to $f^*$ with respect to the pairing $\cup_{\tr}$ in \ref{t2.6}. By construction $f_*$ respects the eigenspaces of $\theta$ and hence it commutes with $\theta$. Consider the defining equation
\begin{equation}
\label{eq:4}
f^* (h) \cup_{\tr} h' = h \cup_{\tr} f_* (h') \; ,
\end{equation}
where $h$ and $h'$ are finite sums of eigenvectors in $H^1 (\Yh_K , \Ch)$ and $H^1 (\Yh_L , \Ch)$. Substituting $* h$ for $h$ and using that $f^*$ commutes with $*$, we get
\[
* f^* (h) \cup_{\tr} h' = * h \cup_{\tr} f_* (h') \quad \text{and hence} \quad \langle h' , f^* (h) \rangle = \langle f_* (h') , h \rangle \; .
\]
Since $\langle, \rangle$ is hermitian, this implies
\begin{equation}
\label{eq:5}
\langle f^* (h) , h' \rangle = \langle h , f_* (h') \rangle \; .
\end{equation}
It follows that if we replace $H^1 (\Yh_K , \Ch)$ with its Hilbert space completion with respect to $\langle , \rangle$, then $f_*$ is the Hilbert space adjoint of $f^*$. Note here that by diagram \eqref{eq:3} we have
\begin{align}
\langle f^* (h_1) , f^* (h_2) \rangle & = \tr f^* (h_1 \cup * h_2) = [L : K] \tr (h_1 \cup *h_2) \nonumber \\
& = [L : K] \langle h_1 , h_2 \rangle \quad \text{for} \; h_1 , h_2 \in H^1 (\Yh_K , \Ch) \; . \label{eq:6}
\end{align}
Hence $[L : K]^{- 1/2} f^* : H^1 (\Yh_K , \Ch) \to H^1 (\Yh_L , \Ch)$ is an isometry and in particular $f^*$ is bounded for the norm corresponding to $\langle , \rangle$. For the alternating pairing $\cup_{\tr}$ in \ref{t2.6} we have by the same argument
\begin{equation}
\label{eq:7}
f^* (h_1) \cup_{\tr} f^* (h_2) = [L : K] h_1 \cup_{\tr} h_2 \; .
\end{equation}
It follows from either \eqref{eq:4} + \eqref{eq:7} or \eqref{eq:5} + \eqref{eq:6} that we have $f_* f^* = [L : K]$ on $H^1 (\Yh_K , \Ch)$. 
\end{punkt}
\begin{punkt} \label{t2.9} \em 
The discussion in \ref{t2.8} implies in particular that via $\sigma \mapsto (\overline{\spec} \sigma)^*$ the group $\Aut (K)$ acts from the left on $H^1 (\Yh_K , \Ch)$ respecting both $\langle , \rangle$ and $\cup_{\tr}$, hence $\Aut (K)$ acts by isometric symplectomorphisms. Moreover, we have $(\overline{\spec} \sigma)_* = (\overline{\spec} \sigma^{-1})^*$. The $\Aut (K)$-action commutes with the endomorphism $\theta$ on $H^1 (\Yh_K , \Ch)$ and hence respects its eigenspaces. We remark that on the Hilbert space completion, the operator $\theta$ is unbounded because its eigenvalues $\rho$, the zeroes of $\hzeta_K (s)$ are unbounded.
\end{punkt}
\begin{punkt} \label{t2.10} \em 
Similar arguments as e.g. for the geometric \'etale cohomology of curves over finite fields work as well in the conjectural cohomological formalism and imply that for a Galois extension of number fields $L / K$ with group $G$, we have 
\begin{equation} \label{eq:9}
f^* f_* = \sum_{\sigma \in G} \sigma \quad \text{on} \; H^1 (\Yh_L , \Ch) \; .
\end{equation}
Here we have written $\sigma$ for the action by $(\overline{\spec} \sigma)^*$ on $H^1 (\Yh_L , \Ch)$.
\end{punkt}
\begin{punkt} \label{t2.12} \em 
Consider the $\halb$-eigenspace $(H^1 (\Yh_K , \Ch))^{\theta = 1/2}$. Serre's conjecture from the introduction and formula \eqref{eq:2n} imply that for $K / \Q$ Galois, we should have 
\begin{equation}
\label{eq:10}
\dim H^1 (\Yh_K , \Ch)^{\theta = 1/2} = \sum_{\pi \, \text{sympl} \atop W (\pi) = -1} \deg \pi \; .
\end{equation}
Here $\pi$ runs over the symplectic irreducible representations of $G = \Gal (K / \Q)$ with $W (\pi) = -1$. In fact, a somewhat more involved cohomological argument, \cite{D2} 2.13 and equation \eqref{eq:10} imply the following more precise assertion. As $\C [G]$-modules we have
\begin{align*}
H^1 (\Yh_K , \Ch)^{\theta = 1/2} & = \bigoplus_{\pi \, \text{sympl} \atop W (\pi) = -1} H^1 (\Yh_K , \Ch)^{\theta = 1/2}  (\pi) \quad \text{and}\\
\dim H^1 (\Yh_K , \Ch)^{\theta = 1/2} (\pi) & = \deg\pi \; .
\end{align*}
In particular the $\pi$-isotypical components $H^1 (\Yh_K , \Ch)^{\theta = 1/2} (\pi)$ appear with multiplicity one in $H^1 (\Yh_K , \Ch)^{\theta = 1/2}$. 
\end{punkt}
In the following, for simplicity an antilinear endomorphism $*$ of a $\C$-vector space with $*^2 = -1$ will be called a star operator.

\begin{punkt} \label{t2.13} \em
Apart from the conjectural infinite dimensional cohomology theory $H^i (\Yh_K , \Ch)$ with operator $\theta$, there should also be Arakelov motivic cohomology groups $H^i_{\Mh} (\Yh_K , n/2)$ for half integer twists together with regulator maps
\[
r_{i,n} : H^i_{\Mh} (\Yh_K , n / 2) \longrightarrow H^i (\Yh_K , \Ch)^{\theta = n / 2} \; .
\]
For even $n$ such groups have been defined and studied in \cite{Schol}, \cite{HSchol} also in higher dimensions. For $i = 2 = n$ we have
\[
H^2_{\Mh} (\Yh_K , 1) = CH^1 (\Yh_K) \; .
\]
Via this identification $r_{2,2}$ should be the map $cl$ in \eqref{eq:1}. We expect a commutative diagram
\[
\xymatrix{
H^1_{\Mh} (\Yh_K , 1/2) \times H^1_{\Mh} (\Yh_K , 1/2) \ar[d]^{r_{1,1} \times r_{1,1}} \ar[r]^-{\cup} & H^2_{\Mh} (\Yh_K , 1) \ar[d]^{r_{2,2}} \ar@{=}[r] & CH^1 (\Yh_K) \ar[r]^-{\deg} & \R \ar@{_{(}->}[d] \\
H^1 (\Yh_K , \Ch)^{\theta = 1/2} \times H^1 (\Yh_K , \Ch)^{\theta = 1/2} \ar[r]^-{\cup} & H^2 (\Yh_K , \Ch)^{\theta = 1} \ar[rr]^-{\tr} && \C \; .
}
\]
There should also be Arakelov motivic cohomology $\C$-vector spaces with half integer twists $H^i_{\Mh} (\Yh_K , \C (n/2))$ and a factorization of the regulator map $r_{i,n}$ as follows
\[
r_{i,n} : H^i_{\Mh} (\Yh_K , n/2) \longrightarrow H^i_{\Mh} (\Yh_K , \C (n/2)) \longrightarrow H^i (\Yh_K , \Ch)^{\theta = n/2} \; .
\]
For even $n$, the vector spaces $H^i_{\Mh} (\Yh_K , \C (n/2))$ exist even for higher dimensional arithmetic schemes as the complexification of the real Arakelov motivic cohomology vector spaces studied in \cite{Schol}. We expect that the regulator map
\[
H^1_{\Mh} (\Yh_K , \C (1/2)) \silo H^1 (\Yh_K , \Ch)^{\theta = 1/2}
\]
is an isomorphism. 
\end{punkt}

A Birch and Swinnerton-Dyer conjecture for the Dedekind zeta function would consist in finding $H^1_{\Mh} (\Yh_K , 1/2)$ and its cousin $H^1_{\Mh} (\Yh_K , \C (1/2))$ and use them to describe the vanishing order of $\zeta_K (s)$ at $s = 1/2$ (as the dimension of $H^1_{\Mh} (\Yh_K , \C (1/2))$) and the leading coefficient $\zeta^*_K (1/2)$ in the Taylor expansion of $\zeta_K (s)$ at $s = 1/2$. For $\zeta^*_K (1/2)$ the pairing $\cup$ on $H^1_{\Mh} (\Yh_K , 1/2)$ might be involved as well. Like for Arakelov motivic cohomology for even twists, any definition of $H^1_{\Mh} (\Yh_K , 1/2)$ or $H^1_{\Mh} (\Yh_K , \C (1/2))$ should involve a mixture of number theory and analysis. In section \ref{sec:4} we discuss exponential motives in this context. 

The cohomological considerations in \ref{t2.1}--\ref{t2.12} together with Serre's vanishing conjecture suggest the following Prediction \ref{t2.14}, where $\Nh$ is the  category of number fields with ring homomorphisms $\alpha : K \to L$ as morphisms and where we set $\alpha_* = f^*$ and $\alpha^* = f_*$ for $f = \overline{\spec} \alpha$. We apologize for listing the above properties of the groups $H^1 (\Yh_K , \Ch)^{\theta = 1/2}$ again for the groups $H^1_{\Mh} (\Yh_K , \C (1/2))$ which should be isomorphic. Neither of these cohomology theories has yet been defined but we think of them as being of a very different nature. In the analogous situation for an elliptic curve $E / \Q$ and the central points $s = 1$ the analogue of $H^1_{\Mh} (\Yh_K , \C (1/2))$ is $CH^1 (E)^0 \otimes \C = E (\Q) \otimes \C$ and the scalar hermitian product is the positive definite version of the height pairing. However no analogue of $H^1 (\Yh_K , \Ch)$ for elliptic curves has been constructed. 

By a co- and contravariant functor $F$ from a category $\Ch$ to a category $\Dh$ we mean the following:\\
1) A map $F : \Ob (\Ch) \to \Ob (\Dh)$.\\
2) For all objects $X$ and $Y$ of $\Ch$, maps $F_* = F_* (X,Y) : \Mor_{\Ch} (X,Y) \to \Mor_{\Dh} (F (X) , F (Y))$ such that $(F , F_*)$ is a covariant functor.\\
3) For all objects $X$ and $Y$ of $\Ch$, maps $F^* = F^* (X,Y) : \Mor_{\Ch} (X,Y) \to \Mor_{\Dh} (F (Y) , F (X))$ such that $(F , F^*)$ is a contravariant functor.

\begin{prediction} \label{t2.14} \em 
There is a co- and contravariant ``motivic cohomology'' functor from the category of number fields $\Nh$ to the category $\nVec_{\C}$ of finite-dimensional $\C$-vector spaces
\[
K \longmapsto H^1_{\Mh} (\Yh_K , \C (1/2)) \; , \; \alpha \longmapsto \alpha_* \; , \; \alpha^*
\]
with the following properties:\\
1) $\alpha^* \alpha_* = [L : K]$ on $H^1_{\Mh} (\Yh_K , \C (1/2))$ for $\alpha : K \hookrightarrow L$\\
2) If $L / K$ is Galois with group $G$, then
\[
\alpha_* \alpha^* = \sum_{\sigma \in G} \sigma_* \quad \text{on} \; H^1_{\Mh} (\Yh_L , \C (1/2)) \; \text{for any} \; \alpha : K \hookrightarrow L \; .
\]
3) If $K / \Q$ is Galois with group $G$, then we have as $\C [G]$-modules
\[
H^1_{\Mh} (\Yh_K , \C (1/2)) = \bigoplus_{\pi \, \text{sympl} \atop W (\pi) = -1} H^1_{\Mh} (\Yh_K , \C (1/2)) (\pi)
\]
and $\dim H^1_{\Mh} (\Yh_K , \C (1/2)) (\pi) = \deg \pi$. \\
Here $\pi$ runs over the isomorphism classes of complex irreducible symplectic representations of $G$. In particular
\[
\dim H^1_{\Mh} (\Yh_K , \C (1/2)) = \sum_{\pi \, \text{sympl} \atop W (\pi) = - 1} \deg \pi \; .
\]
4) The vector space $H^1_{\Mh} (\Yh_K  , \C (1/2))$ carries a symplectic pairing
\[
\cup_{\tr} : H^1_{\Mh} (\Yh_K , \C (1/2)) \times H^1_{\Mh} (\Yh_K , \C (1/2)) \longrightarrow \C
\]
and an antilinear operator $*$ with $*^2 = -1$ such that
\[
\langle h_1 , h_2 \rangle = h_1 \cup_{\tr} (* h_2)
\]
defines a hermitian scalar product on $H^1_{\Mh} (\Yh_K , \C (1/2))$. \\
5) For $\alpha : K \hookrightarrow L$ we have\\
a) $\alpha_* \verk * = * \verk \alpha_* : H^1_{\Mh} (\Yh_K , \C (1/2)) \to H^1_{\Mh} (\Yh_L , \C (1/2))$\\
b) $\alpha^*$ is adjoint to $\alpha_*$ via $\cup_{\tr}$ and hence via $\langle , \rangle$, that is
\[
\alpha_* h \cup_{\tr} h' = h \cup_{\tr} \alpha^* h' \quad \text{and} \quad \langle \alpha_* h , h' \rangle = \langle h , \alpha^* h' \rangle
\]
for $h \in H^1_{\Mh} (\Yh_K , \C (1/2))$ and $h' \in H^1_{\Mh} (\Yh_L , \C (1/2))$. \\
Note that the condition that the scalar product in 4) is hermitian is equivalent to the formula
\[
*h_1 \cup_{\tr} *h_2 = \overline{h_1 \cup_{\tr} h_2} \quad \text{for} \; h_1 , h_2 \in H^1_{\Mh} (\Yh_K , \C (1/2)) \; .
\]
c) $\alpha_* (h_1) \cup_{\tr} \alpha_* (h_2) = [L : K] h_1 \cup_{\tr} h_2$ and hence $\langle \alpha_* (h_1) , \alpha_* (h_2) \rangle  = [L : K] \langle h_1 , h_2 \rangle$. \\
\end{prediction}

\begin{punkt} \label{t2.16} \em
Serre's vanishing conjecture which a priori has nothing to do with cohomology fits very well with the expected existence of a cup-product on $H^1 (\Yh_K , \Ch)$ and hence on 
\[
H^1_{\Mh} (\Yh_K , \C (1/2)) \cong H^1 (\Yh_K , \Ch)^{\theta = 1/2} \; .
\]
Namely, as noted in \cite{D2} \S\,2.13, since the $H^1_{\Mh} (\Yh_K , \C (1/2)) (\pi)$'s are irreducible, self-dual and pairwise non-isomorphic the restriction of 
\[
\cup_{\tr} : H^1_{\Mh} (\Yh_K , \C (1/2)) \times H^1_{\Mh} (\Yh_K , \C (1/2)) \to \C
\]
to $H^1_{\Mh} (\Yh_K , \C (1/2)) (\pi) \times H^1_{\Mh} (\Yh_K , \C (1/2)) (\pi)$ must be the (up to scalar) unique $G$-invariant symplectic pairing on $H^1_{\Mh} (\Yh_K , \C (1/2)) (\pi)$. Hence given 1), 2), 4), 5) we can phrase 3) equivalently as follows:\\
3') If $K / \Q$ is Galois with group $G$ consider the action of $\sigma \in G$ by the symplectic isomorphism $\sigma_*$ on $(H^1_{\Mh} (\Yh_K , \C (1/2)) , \cup_{\tr})$. The canonical decomposition of $H^1_{\Mh} (\Yh_K , \C (1/2))$ into isotypical components $H^1_{\Mh} (\Yh_K , \C (1/2)) (\pi)$ for the $G$-action has the following properties:\\
a) The irreducible complex representations of $G$ that occur in $H^1_{\Mh} (\Yh_K , \C (1/2))$ have multiplicity one, and hence
\[
\dim H^1_{\Mh} (\Yh_K , \C (1/2)) (\pi) = \deg \pi \quad \text{if} \; H^1_{\Mh} (\Yh_K , \C (1/2)) (\pi) \neq 0 \; .
\]
b) The restriction of the $G$-invariant pairing $\cup_{\tr}$ on $H^1_{\Mh} (\Yh_K , \C (1/2))$ to each non-zero isotypical component $H^1_{\Mh} (\Yh_K , \C (1/2)) (\pi)$ is non-degenerate. In particular, only symplectic $\pi$ appear in $H^1_{\Mh} (\Yh_K , \C (1/2))$.\\
c) The root number of each $\pi$ appearing in $H^1_{\Mh} (\Yh_K , \C (1/2))$ is $W (\pi) = -1$. 
\end{punkt}

A priori it is not clear that a co- and contravariant functor with properties as in the prediction exists. However, because of the multiplicity one condition in 3) the situation is quite rigid and we can prove existence and essential uniqueness of functors as in the prediction. To do so it is convenient to rescale $\cup_{\tr}$ on $H^1_{\Mh} (\Yh_K , \C (1/2))$ by setting 
\[
\cup_K = [K : \Q]^{-1} \cup_{\tr} \; .
\]
Then $\alpha_*$ respects the rescaled symplectic pairings. Moreover $[L : K]^{-1} \alpha^*$ is adjoint to $\alpha_*$. Let $\nVec^{\sharp}_{\C}$ be the category of finite dimensional $\C$-vector spaces $W$ with a symplectic pairing $\cup : W \times W \to \C$ and a star operator $*$ such that the formula $\langle w, w' \rangle = w \cup * w'$ defines a hermitian scalar product on $W$. The property $\langle \overline{w , w'} \rangle = \langle w' , w \rangle$ for $w , w' \in W$ is equivalent to the relation $\overline{w_1 \cup w_2} = * w_1 \cup *w_2$ for $w_1 , w_2 \in W$. The morphisms in $\nVec^{\sharp}_{\C}$ are $\C$-linear maps $\varphi : W \to W'$ which respect $\cup$ and commute with $*$. In particular they are injective and isometric. Let $\tvarphi : V' \to V$ be the adjoint of $\varphi$ with respect to the symplectic pairings on $V$ and $V'$ or equivalently with respect to the scalar products. Prediction \ref{t2.14} is equivalent to the conjectural functor 
\[
\Nh \longrightarrow \nVec^{\sharp}_{\C} \; , \; K \longmapsto (H^1_{\Mh} (\Yh_K , \C (1/2)) , \cup_K , *) \; , \; \alpha \longmapsto \alpha_*
\]
satisfying the conditions on $V$ in the following result.

\begin{theorem}
\label{t2.15}
There exists a covariant functor
\[
V : \Nh \longrightarrow \nVec^{\sharp}_{\C} \; , \; K \longmapsto V (K) = (V_K , \cup_K , *_K) \; , \; \alpha \longmapsto V (\alpha)
\]
with the following properties, where $\tV (\alpha)$ is the adjoint of $V (\alpha)$\\
1) $\tV (\alpha) V (\alpha) = \id$ on $V_K$ for $\alpha : K \hookrightarrow L$.\\
2) If $L / K$ is Galois with group $G$, then
\[
V (\alpha) \tV (\alpha) = \frac{1}{[L : K]} \sum_{\sigma \in G} V (\sigma) \; \text{on} \; V_L \; \text{for any} \; \alpha : K \hookrightarrow L \; .
\]
3) If $K / \Q$ is Galois with group $G$, then as $\C [G]$-modules
\begin{equation} \label{eq:11n}
V_K = \bigoplus_{\pi \, \text{sympl}\atop W (\pi) = -1} V_K (\pi) \quad \text{and} \quad \dim V_K = \deg \pi \; .
\end{equation}
\end{theorem}

Let $\oQ \subset \C$ be the algebraic closure of $\Q$ in $\C$ and let $G_{\Q} = \Aut (\oQ)$ be the absolute Galois group of $\Q$. Let $\Pi$ be the set of isomorphism classes of irreducible symplectic continuous representations $\pi$ of $G_{\Q}$ on finite dimensional $\C$-vector spaces with $W (\pi) = -1$. Consider the group $\Map ( \Pi , \mu_2)$ of maps from the set $\Pi$ to $\mu_2 = \{ \pm 1 \}$. 

\begin{theorem}
\label{t2.17}
1) Any two functors $V : \Nh \to \nVec^{\sharp}_{\C}$ in Theorem \ref{t2.15} are isomorphic.\\
2) The automorphism group of any $V$ is isomorphic to $\Map (\Pi , \mu_2)$.
\end{theorem}

\textbf{Remark} For a tame Galois extension $K / \Q$ with group $G$, the symplectic irreducible representations with root number $-1$ are the obstructions for the projective $\Z [G]$-module $\eo_K$ to be zero in the stable class group $Cl (\Z [G])$, \cite{T}. This could be a hint that the Galois module structure of $\eo_K$ may play a role in a \textit{natural} construction of a functor $V$ as in Theorem \ref{t2.15}, a construction not depending on choices as in our purely formal existence proof below.
\section{Proofs} \label{sec:3}

We need the following elementary result:

\begin{prop}
\label{t3.1}
Let $V$ be a finite dimensional irreducible complex representation of a finite group $G$ equipped with a $G$-invariant symplectic pairing $\cup : V \times V \to \C$. Then there is a unique star operator $*$ on $V$ such that $\langle v, w \rangle = v \cup *w$ for $v,w \in V$ defines a hermitian scalar product on $V$.
\end{prop}

\begin{proof}
If $*_1$ and $*_2$ are two such star operators, the composition $*_1 \verk *^{-1}_2$ is a $G$-equivariant $\C$-linear endomorphism of $V$, hence a scalar by Schur's Lemma. Thus $*_1 = \mu *_2$ for some $\mu \in \C$. We have $|\mu| = 1$ since
\[
-1 = *^2_1 = (\mu *_2)^2 = |\mu|^2 *^2_2 = - |\mu|^2 \; .
\]
On the other hand we have $\mu > 0$ because of the relations $0 < v \cup *_1 v = \mu (v \cup *_2 v)$ and $0 < v \cup *_2 v$ for $0 \neq v \in V$. Thus $\mu = 1$ and uniqueness follows. \\
For existence, choose a $G$-invariant hermitian scalar product $\langle , \rangle$ on $V$ and define an anti-linear $G$-equivariant automorphism $*$ of $V$ by the formula
\[
\langle v,w \rangle = v \cup *w \quad \text{for} \; v,w \in V \; .
\]
Then $*^2$ is a $G$-equivariant $\C$-linear endomorphism of $V$ and hence a scalar, $* = \lambda \, \id$. For $0 \neq v \in V$ we have $* v \neq 0$ as well and therefore
\[
0 < \langle *v , *v \rangle = *v \cup *^2v = - \lambda v \cup *v = - \lambda \langle v,v \rangle \; .
\]
It follows that $\lambda < 0$ and replacing $*$ by $|\lambda|^{-1/2} *$ we get a star operator as in the proposition. 
\end{proof}

\begin{proofof} \textbf{Theorem \ref{t2.15}} 
As before, let $\oQ$ be the algebraic closure of $\Q$ in $\C$ and let $\Nh_e$ be the category of embedded subfields $K \subset \oQ$ which are finite extensions of $\Q$. A morphism from $K \subset \oQ$ to $L \subset \oQ$ is a homomorphism of fields $K \to L$. It does not have to be compatible with the inclusions of $K$ and $L$ into $\C$. The forgetful functor $\Nh_e \to \Nh$ is an equivalence of categories. Choosing a quasi-inverse, it suffices to prove Theorem \ref{t2.15} with $\Nh$ replaced by $\Nh_e$. Note that if we have two homomorphisms $\alpha_1 : K \hookrightarrow L$ and $\alpha_2 : K \hookrightarrow L$ then $L$ is Galois over $\alpha_1 (K)$ if and only if it is Galois over $\alpha_2 (K)$. We leave the notations as before but from now on, every field $K$ is a subfield of $\oQ$ and therefore equipped with its inclusion map $K \subset \oQ$. For $K \subset \oQ$ let $G_K = \Gal (\oQ / K)$ be the corresponding open subgroup of $G_{\Q} = \Gal (\oQ / \Q)$. For every $\pi$ in the set $\Pi$ defined above choose a representing vector space $V_{\pi}$ and a $G_{\Q}$-invariant symplectic form $\cup_{\pi} : V_{\pi} \times V_{\pi} \to \C$. We equip the complex vector space $V_{\oQ} = \bigoplus_{\pi \in \Pi} V_{\pi}$ with the alternating form $\cup : V_{\oQ} \times V_{\oQ} \to \C$ which is the orthogonal direct sum of the $\cup_{\pi}$'s. The group $G_{\Q}$ acts with finite orbits on $V_{\oQ}$. Equivalently, the representation of $G_{\Q}$ on $V_{\oQ}$ is smooth in the sense that $V_{\oQ} = \bigcup_{K} (V_{\oQ})^{G_K}$ where $K \in \Nh_e$. Moreover, the representation is also admissible, meaning that $V_K = (V_{\oQ})^{G_K}$ is finite dimensional for all $K$ in $\Nh_e$. To see this, we may assume that $K / \Q$ is Galois. Then $G_K$ is a normal subgroup of $G_{\Q}$ and hence $V^{G_K}_{\pi}$ is a $G_{\Q}$-invariant subspace of $V_{\pi}$. Hence we have $V^{G_K}_{\pi} = 0$ unless $G_K$ acts trivially on $V_{\pi}$ in which case $V^{G_K}_{\pi} = V_{\pi}$. Hence we have
\begin{equation}
\label{eq:11}
V_K = \bigoplus_{ \pi \in \Pi_K} V_{\pi} \; .
\end{equation}
Here $\Pi_K \subset \Pi$ consists of those symplectic $\pi$ with $W (\pi) = -1$ that factor over $\Gal (K / \Q) = G_{\Q} / G_K$. In particular, $V_K$ is finite dimensional and satisfies property 3) in Theorem \ref{t2.15}. The restriction $\cup_K$ of $\cup$ to $V_K \times V_K$ is a symplectic pairing since it is the orthogonal direct sum of the symplectic pairings $\cup_{\pi}$ for $\pi$ in $\Pi_K$. 

Applying Proposition \ref{t3.1} to $(V_{\pi} , \cup_{\pi})$ we obtain a $G_{\Q}$-equivariant star operator $*_{\pi}$ on $V_{\pi}$. The direct sum of these operators gives a $G_{\Q}$-equivariant $*$-operator on $V_{\oQ}$ for which $\langle v_1 , v_2 \rangle = v_1 \cup *v_2$ is a hermitian scalar product on $V_{\oQ}$. Its restriction to $V_K$ is denoted by $*_K$ and it equals the direct sum of the $*_{\pi}$ for $\pi \in \Pi_K$. 

We can now construct a functor $V: \Nh_e \to \Vh^{\sharp}$. For $K$ an object of $\Nh_e$, an embedded subfield $K \subset \oQ$ we set
\[
V (K) = (V_K , \cup_K , *_K) \; .
\]
For a morphism in $\Nh_e$, $\alpha : K \to L$ choose a prolongation of $\alpha$ to a homomorphism $\oalpha : \oQ \to \oQ$ along the given embeddings of $K$ and $L$ into $\oQ$. This gives a commutative diagram
\[
\xymatrix{
\oQ \ar[r]^{\oalpha} \ar@{}[d]|{\cup} & \oQ \ar@{}[d]|{\cup}\\
K \ar[r]^{\alpha} & L \; .
}
\]
The injection $\oalpha$ is actually an automorphism, $\oalpha \in G_{\Q}$. It induces an automorphism $\oalpha : V_{\oQ} \to V_{\oQ}$ which maps $V_K$ into $V_L$ since $\oalpha^{-1} G_L \oalpha \subset G_K$. The induced $\C$-linear map
\[
V (\alpha) = \oalpha \, |_{V_K} : V_K \longrightarrow V_L
\]
depends only on $\alpha$ and not on the choice of $\oalpha$: If $\ooalpha$ is another prolongation, then $\oalpha^{-1} \verk \ooalpha$ fixes $K$ and hence $\ooalpha = \oalpha \verk \sigma$ for some $\sigma \in G_K$. Thus $\ooalpha \, |_{V_K} = \oalpha \, |_{V_K}$. The alternating form $\cup$ respectively the star operator $*$ on $V_{\oQ}$ are $G_{\Q}$-invariant respectively equivariant. Since $\cup_K , \cup_L$ and $*_K , *_L$ are the restrictions of $\cup$ and $*$ to $V_K$ respectively $V_L$ it follows that for $v_1 , v_2 , v \in V_K$ we have
\[
\oalpha (v_1) \cup_L \oalpha (v_2) = v_1 \cup_K v_2 \quad \text{and} \quad *_L (\oalpha (v)) = \oalpha (*_K (v)) \; .
\]
This means that $V (\alpha)$ is a morphism from $V (K)$ to $V (L)$ in $\Vh^{\sharp}$. It is clear that $V : \Nh_e \to \Vh^{\sharp}$ is a covariant functor. The adjoint map $\tV (\alpha) : V_L \to V_K$ to $V (\alpha) : V_K \to V_L$ is defined by 
\[
V (\alpha) v \cup_L w = v \cup_K \tV (\alpha) w \quad \text{for} \; v \in V_K \, , \, w \in V_L \; . 
\]
Hence, using the above compatibility of $\oalpha$ with $\cup$, we have
\[
v \cup_K v' = V (\alpha) v \cup_L V (\alpha) v' = v \cup_K (\tV (\alpha) \verk V (\alpha)) v' \quad \text{for} \; v , v' \in V_K \; .
\]
This implies that
\[
\tV (\alpha) \verk V (\alpha) = \id \; .
\]
Thus if $\alpha' : K \silo K'$ is an isomorphism we have $\tV (\alpha') = V (\alpha')^{-1}$ and hence $V (\alpha') \verk \tV (\alpha') = \id$ as well. Any embedding $\alpha : K \hookrightarrow L$ can be factored as $\alpha : K \xrightarrow{\alpha'} \alpha (K) \overset{i}{\hookrightarrow} L$ where $i$ is the inclusion and $\alpha' = \alpha$ with the new image $\alpha (K)$ instead of $L$. Since $\alpha'$ is an isomorphism we get
\[
V (\alpha) \verk \tV (\alpha) = V (i) \verk V (\alpha') \verk \tV (\alpha') \verk \tV (i) = V (i) \verk \tV (i) \; .
\]
In order to show that
\begin{equation}
\label{eq:12}
V (\alpha) \verk \tV (\alpha) = \frac{1}{[L : K]} \sum_{\sigma \in G} V (\sigma) =: V (e)
\end{equation}
if $L / K$ is Galois we may therefore assume that $\alpha = i$ is compatible with the given embeddings into $\oQ$, so that we have a commutative diagram
\[
\xymatrix{
K \ar@{^{(}->}[r]^i \ar@{}[d]|{\cap} & L \ar@{}[d]|{\cap}\\
\oQ \ar@{=}[r] & \oQ \; .
}
\]
Then $V (i)$ and $\tV (i)$ have the following easy description. We have $\Pi_L = \Pi_K \cupdot \Pi_{L \setminus K}$ where $\Pi_{L \setminus K}$ consists of those $\pi$ in $\Pi_L$ with $G_K \subsetneqq \Ker \pi$, hence those for which the action of $G_{\Q}$ on $V_{\pi}$ factors over $G = \Gal (L / K)$ and is non-trivial. Setting $W = \bigoplus_{\pi \in \Pi_{L \setminus K}} V_{\pi}$ the map
\[
V (i) : V_K \hookrightarrow V_L = V_K \oplus W
\]
is the inclusion $v \mapsto (v , 0)$. The symplectic form $\cup_L$ is the orthogonal direct sum of $\cup_K$ and $\cup_{L \setminus K}$ the latter being the orthogonal direct sum of the $\cup_{\pi}$ for $\pi \in \Pi_{L \setminus K}$. We have a commutative diagram
\[
\xymatrix{
V_L \ar[d]_{\tV (i)} \ar@{=}[r] & V_K \oplus W \ar[rr]^{\cup_L = \cup_K \oplus \cup_{L \setminus K}} & & V^*_K \oplus W^* \ar@{->>}[d]^{V (i)^*} \\
V_K \ar[rrr]^{\cup_K} & & & V^*_K \; .
}
\]
The dual $V (i)^*$ of the inclusion $V (i)$ is the projection to $V^*_K$. Hence \\
$\tV (i) : V_L = V_K \oplus W \to V_K$ is the projection and we have
\[
(V (i) \verk \tV (i)) (v,w) = (v, 0) \quad \text{for} \; v \in V_K , w \in W \; .
\]
It remains to show that $V (e) (v,w) = (v, 0)$, where $V (e) \in \End (V_L)$ was defined in \eqref{eq:12}. On each $V_{\pi}$ for $\pi$ in $\Pi_L$ the endomorphism $V (e)$ is a projector to $V^G_{\pi}$. Hence $V (e)$ is the identity on $V_K$ and the zero map on $W$. Note that $V^G_{\pi} \neq V_{\pi}$ for $\pi \in \Pi_{L\setminus K}$ and hence $V^G_{\pi} = 0$ since $\pi$ is irreducible. 
\end{proofof}

\begin{proofof} \textbf{Theorem \ref{t2.17}}
Since $\Nh_e \to \Nh$ is an equivalence of categories, it suffices to show that any two functors $V , V' : \Nh_e \to \nVec^{\sharp}_{\C}$ satisfying conditions 1)--3) in Theorem \ref{t2.15} are isomorphic and have automorphism groups isomorphic to $\Map (\Pi , \mu_2)$. For such a functor $V : \Nh_e \to \nVec^{\sharp}_{\C}$ consider the filtered  colimit in $\Ind \nVec^{\sharp}_{\C}$:
\begin{equation}
\label{eq:13}
V (\oQ) = \colim_{K \subset \oQ} V (K) \; .
\end{equation}
Here the index poset consists of the objects $(K \subset \oQ)$ of $\Nh_e$ with $K / \Q$ Galois, ordered by those homomorphisms $i : K \hookrightarrow L$ which are compatible with the inclusions $K \subset \oQ$ and $L \subset \oQ$. The transition maps are $V (i) : V (K) \to V (L)$. Since $V$ is a functor, the group $G_{\Q}$ acts on the object $V (\oQ)$. The $\ind$-category $\Ind \nVec^{\sharp}_{\C}$ can be identified with the category of complex vector spaces with a non-degenerate alternating pairing $\cup$ and a $*$-operator for which $(v_1 , v_2) \mapsto v_1 \cup * v_2$ is a hermitian scalar product. Thus we may write $V (\oQ) = (V_{\oQ} , \cup , *)$ where $V_{\oQ}$ is the underlying vector space of the object $V (\oQ)$ in $\Ind \nVec^{\sharp}_{\C}$. The action of $G_{\Q}$ on $V_{\oQ}$ has finite orbits. All transition maps $V (i)$ are injective because of condition 1) in Theorem \ref{t2.15}. For any subfield $K$ in $\Nh_e$ we therefore have a natural inclusion $V_K \hookrightarrow V_{\oQ}$. By functoriality it is $G_K$-equivariant and hence
\begin{equation}
\label{eq:14}
V_K \hookrightarrow V^{G_K}_{\oQ} = \colim_{L \subset \oQ} V^{G_K}_L \; .
\end{equation}
In the colimit we may restrict to extension fields $K \subset L \subset \oQ$ of $K \subset \oQ$ which are Galois over $\Q$ and we set $G = \Gal (L / \Q)$. Properties 1) and 2) in Theorem \ref{t2.15} for $V : \Nh_e \to \nVec^{\sharp}_{\C}$ imply that the homomorphism $i : K \hookrightarrow L$ induces an isomorphism $V_K = V^{G_K}_L$. Namely, by property 1) the map $\tV (i)$ is surjective and we get
\[
\Imm V (i) = \Imm V (i) \verk \tV (i) \overset{2)}{=} V^G_L = V^{G_K}_L \; .
\]
Hence \eqref{eq:14} is an isomorphism of complex vector spaces. Since the colimit \eqref{eq:13} was taken in $\Ind \nVec^{\sharp}_{\C}$ it follows that the functor $V : \Nh_e \to \nVec^{\sharp}_{\C}$ is canonically isomorphic to the functor sending $K \subset \oQ$ to $V^{G_K}_{\oQ}$ equipped with the restrictions of $\cup$ and $*$ of $V_{\oQ}$. The fact that the restriction of $\cup$ remains non-degenerate can be seen directly because the hermitian scalar product $\_\_ \cup * \_\_ $ remains a hermitian scalar product after restriction. Next we note that the representation of $G_{\Q}$ on $V_{\oQ}$ is smooth by construction and admissible because $V^{G_K}_{\oQ} = V_K$ is finite dimensional for all $K$. It follows that $V_{\oQ}$ is the direct sum of irreducible representations each occuring with finite multiplicity \cite[II.1.5. Proposition]{C}. All these multiplicities have to be one because otherwise we would find a Galois extension $K / \Q$ in $\Nh_e$ for which $V_K = V^{G_K}_{\oQ}$ has an irreducible $G = G_{\Q} / G_K$-representation of multiplicity at least $2$ contradicting condition 3) in Theorem \ref{t2.15}. Again using 3) we see that there is an isomorphism
\[
\varepsilon : V_{\oQ} \cong \bigoplus_{\pi \in \Pi} V_{\pi}
\]
as $G_{\Q}$-representations where the set $\Pi$ was defined before Theorem \ref{t2.17}. For the symplectic form $\cup_{\pi}$ on $V_{\pi}$ we take the one corresponding to the restriction $\cup : V_{\oQ} \times V_{\oQ} \to \C$ to $V_{\oQ} (\pi) \times V_{\oQ} (\pi)$, noting that $\varepsilon$ induces an isomorphism $\varepsilon : V_{\oQ} (\pi) \silo V_{\pi}$. We can transport the star operator on $V_{\oQ}$ via $\varepsilon$ or note that by Proposition \ref{t3.1} it is already uniquely determined by the $\cup_{\pi}$'s. Since the functor $V$ can be recovered from $V (\oQ)$ in $\Ind \nVec^{\sharp}_{\C}$ with the $G_{\Q}$-action it follows that $V$ is isomorphic to a functor of the type constructed in the proof of Theorem \ref{t2.15}. Hence all functors $V$ are isomorphic. Any automorphism  of $V$ gives rise to a $G_{\Q}$-equivariant automorphism of $V_{\oQ}$ which has to respect the $\pi$-isotypical components. Since the latter are irreducible the automorphism acts by scalar multiplication on them. This scalar $\varphi (\pi)$ has to be $\pm 1$ since the symplectic form is preserved. Hence any automorphism of $V$ is determined by a map $\varphi : \Pi \to \mu_2$ and we obtain an injective homomorphism of groups $\Aut (V) \to \Map (\Pi , \mu_2)$. \\
On the other hand, a map $\varphi : \Pi \to \mu_2$ induces a $G_{\Q}$-equivariant automorphism $\varphi$ of the triple $V (\oQ)$. Any endomorphism $f$ of the functor $V$ induces a $G_{\Q}$-equivariant endomorphism of $V_{\oQ}$. It therefore respects the isotypical components of $V_{\oQ}$ and since $\varphi$ acts on these by multiplication with $\pm 1$ it follows that on $V_{\oQ}$ the endomorphism $f$ commutes with the automorphism $\varphi$. The same is true after restriction to $V_K = V^{G_K}_{\oQ}$. Hence $\varphi$ gives a natural transformation $V \to V$ and hence an automorphism of $V$. It follows that the map $\Aut (V) \to \Map (\Pi , \mu_2)$ is also surjective.
\end{proofof}
\section{Further remarks} \label{sec:4}

We have discussed hypothetical cohomology theories for $\Yh_K = \overline{\spec} \eo_K$. On the one hand the groups $H^i (\Yh_K , \Ch)$ with operator $\theta$ and on the other hand in \ref{t2.13} the groups $H^i_{\Mh} (\Yh_K , 1/2)$ and the $\C$-vector spaces $H^i_{\Mh} (\Yh_K , \C (1/2))$. In the paper \cite{D1} for every normal scheme $\eX$ of finite type over $\spec \Z$ a connected topological dynamical system $X = \ceX (\C) \times_{\Q^{> 0}} \R^{> 0}$ was constructed. Here $\ceX (\C)$ is a topological space with an action of $(\Q^{> 0} , \cdot)$ where we think of the action by $p \in \Q^{> 0}$ as a Frobenius at $p$. The group $\R$ acts on $X$ by multiplication via $\exp$ on the second factor. The closed orbits of the $\R$-action on $X$ are in a correspondence (many to one) with the closed points of $\eX$. The space $X$ is equipped with the sheaf $\Ch$ of continuous $\C$-valued functions on $X$ which are smooth in the $\R^{> 0}$-coordinate and locally constant in the $\ceX (\C)$-coordinate. We can consider the sheaf cohomology groups $H^i (X , \Ch)$ with the induced $\R$-action. For $\eX = \spec \eo_K$ these groups together with the infinitesimal generator $\theta$ of the $\R$-action are the best approximation to the conjectured groups $H^i (\Yh_K , \Ch)$ with operator $\theta$ that we can presently produce. However as explained in \cite{D1} our dynamical systems $X$ and hence their sheaf cohomology need to be improved. Thus we have no good candidate for the $1 / 2$-eigenspace $H^1 (\Yh_K , \Ch)^{\theta = 1/2}$ at the moment. 

We now consider the second speculative cohomology group $H^1 (\Yh_K , 1/2)$ which was discussed in section \ref{sec:2}. For elliptic curves $E / \Q$ the corresponding group is the motivic cohomology group $H^2_{\Mh} (E , \Q (1)) = CH^1 (E)^0 \otimes \Q = E (\Q) \otimes \Q$, because the Birch and Swinnerton-Dyer conjecture asserts that $\rk E (\Q) = \ord_{s = 1} L (E,s)$. In a suitable category of motivic sheaves on $\Yh_K$, we expect an equality of the form $H^2_{\Mh} (E , \Q (1)) = H^1_{\Mh} (\Yh_K , j_{!*} H^1 (E) (1))$ where $j : \spec K \hookrightarrow \Yh_K$ is the inclusion. Correspondingly we think of $H^1_{\Mh} (\Yh_K , 1/2)$ as $H^1 (\Yh_K , j_{!*} H^0 (\spec K) (1/2))$ in an extended category of motivic sheaves over $\Yh_K$ which allow half-integer Tate twists. 

One might try to realize $H^1_{\Mh} (\Yh_K , 1/2)$ as 
\[
\Ext^1_{M^{\exp}_{\Yh_K}} (\Q (0) , \Q (1/2)) \subset \Ext^1_{M^{\exp} (K)} (\Q (0) , \Q (1/2)) \; .
\]
Here $M^{\exp}_{\Yh_K}$ is a full $\Q$-linear subcategory of ``integral'' exponential motives in the $\Q$-linear neutral Tannakian category of exponential motives $M^{\exp} (K)$ over $K$, \cite{FJ}, in particular Ch. 12. Moreover $\Q (1/2)$ is the dual of the exponential motive $\Q (-1/2) = (\A^1_K , f = x^2)$ as defined in \cite[12.2]{FJ}. The expected $\cup$-product should be the Yoneda pairing of this group with itself with values in $\Ext^2_{M_{\Yh_K}} (\Q (0) , \Q (1))$, the $2$-extensions of classical motives which are integral over $\Yh_K$, \cite{DN}. Note that classical motives embed as a full abelian subcategory of exponential motives but their essential image is not stable under extensions. Conditionally this target group has a natural map ``$cl^{-1}$'' to the Arakelov Chow group $CH^1 (\Yh_K) \otimes \Q$ and hence to $\R$, \cite{DN}. There are at least three problems with this idea: \\
1) The exponential motive $\Q (-1/2)$ has square $\Q (-1/2)^{\otimes 2}_K = (\A^2_K , f = x^2_1 + x^2_2)$ in $M^{\exp} (K)$. If the field $K$ contains $i = \sqrt{-1}$, then we indeed obtain $\Q (-1)$. If not, the square of $\Q (-1/2)$ is not $\Q (-1)$ but $M_{\chi} \otimes \Q (-1)$ where $M_{\chi}$ is the motive of the non-trivial character $\chi$ of the quadratic extension $K(i) / K$. Now, according to \cite{O} there are Galois extensions $K / \Q$ with $G = Q_8$ the quaternion group whose unique irreducible symplectic representation $\pi$ has $W (\pi) = -1$. Moreover, a theorem of Witt at the end of \cite{W} describes the quadratic subfields of $Q_8$-extensions $K$ precisely, and they are all real-quadratic. Thus for such $K$ which should have a non-zero $H^1_{\Mh} (\Yh_K , 1/2)$ we do not have a candidate for $\Q (1/2)$ in $M^{\exp} (K)$. \\
2) The extension groups of exponential motives are $\Q$-vector spaces. However, for any $K / \Q$ Galois the $\C$-vector space $H^1_{\Mh} (\Yh_K , \C (1/2))$ derived from $H^1_{\Mh} (\Yh_K , 1/2)$ if non-zero cannot even have a functorial real structure if Prediction \ref{t2.14} is true. This follows from the argument in the proof of \cite{D2} Theorem 2.1. For quaternion extensions it is clear because by 3) of Prediction \ref{t2.14} we would have $H^1_{\Mh} (\Yh_K , \C (1/2)) \cong V_{\pi}$ as a $\C [G]$-module and it is known that irreducible symplectic representations cannot be realized over $\R$. \\
3) Finally, Fres\'an and Jossen proved in Theorem \ref{t4.1} below that the group \\
$\Ext^1_{M^{\exp}(K)} (\Q (0) , \Q (1/2))$ and hence any subgroup $\Ext^1_{M^{\exp}_{M_{\Yh_K}}} (\Q (0) , \Q (1/2))$ vanish. 

It is conceivable that a ``twisted'' version of the category of exponential motives resolves the three problematic issues: $\sqrt{-1} \notin K$, the $\Q$-structure and the vanishing of 
\[
\Ext^1_{M^{\exp} (K)} (\Q (0), \Q (1/2)) \; .
\] 
In any case, guessing a natural space whose dimension is $\ord_{s = 1/2} \zeta_K (s)$ assuming Serre's conjecture remains a challenge even though by Prediction \ref{t2.14} and Theorems \ref{t2.15} and \ref{t2.17} we know its structure abstractly. As for predicting the leading coefficient $\zeta^*_K (1/2)$, the expected positive definite form $\langle , \rangle = \_\_ \cup* \_\_$, an analogue of the height pairing on elliptic curves might play a role. One can also ask if there is a Zagier type conjecture involving values at algebraic integers $z$ of the polylogarithm function $\Li_s (z)$ for $s = 1/2$. 

\begin{theorem}[J. Fres\'an and P. Jossen] \label{t4.1} 
For any subfield $K \subset \C$, we have
\[
\Ext^1_{M^{\exp} (K)} (\Q (0) , \Q (1/2)) = 0 \; .
\]
\end{theorem}

\begin{proof} (from a letter by Fres\'an to the author)
In the first part of the argument, one proves that in the category $\underline{\Perv}_0$ of perverse sheaves with
vanishing cohomology on the complex affine line, every extension
\[
0 \longrightarrow \Q (1/2) \longrightarrow A \longrightarrow \Q(0) \longrightarrow 0
\]
is split. We slightly abuse notation here, using the same symbols $\Q (0)$ and $\Q(1/2)$ for the perverse realisations of these exponential motives. In particular, $\Q (0)$ is denoted by $E(0)$ in \cite{FJ}; it is the perverse sheaf $j_! \Q [1]$, where $j : \Ge_m \hookrightarrow \A^1$ stands for the inclusion. The perverse realisation of $\Q (1/2)$ is the perverse sheaf $j_! L [1]$, where $L$ now stands for the one-dimensional local system on $\Ge_m$ with monodromy $-1$ around $0$ (indeed, the perverse realisation of $H^1 (\A^1 , x^2)$ is obtained by taking the direct image by the function $x \mapsto x^2$ of the constant sheaf $\Q$ on $\A^1$ and quotienting by the copy of $\Q$ inside this direct image; this is what produces the monodromy $-1$). So, what do we know about our perverse sheaf $A$ in the middle of the exact sequence? It has dimension $2$ (in the tannakian sense, meaning generic rank here) and the only singularity is at $0$, where the fibre vanishes (as this is the case for both $\Q (1/2)$ and $\Q (0)$). So the object is entirely determined by its monodromy around $0$, which is a group morphism $\rho : \pi_1 (\C^{\times} , 1) \to \GL_2 (\Q)$. On a basis of the fibre of $A$ at $1$ that is adapted to the exact sequence, the image of the standard generator of the fundamental group of $\C^{\times}$ is a matrix of the form $\left( \begin{smallmatrix} -1 & * \\ 0 & 1 \end{smallmatrix} \right)$, where on the diagonal we find the monodromy on the graded pieces and the entry $*$ gives the class of the extension. But because the eigenvalues are rational and different, all these matrices are diagonalisable over $\Q$, which means that the local system
underlying $A$ is isomorphic to the direct sum of the two rank-one local systems, and hence $A$ itself is isomorphic to the direct sum of $\Q (1/2)$ and $\Q(0)$.

Let us now consider an extension 
\[
0 \longrightarrow \Q (1/2) \longrightarrow M \longrightarrow \Q (0) \longrightarrow 0
\]
in the category of exponential motives over a subfield $K$ of the complex numbers, and let $A = R_{\perv} (M )$ be the perverse realisation of $M$. By the first part of the argument, the splitting of the exact sequence of perverse realisations shows that $A$ contains $\Q (0)$ as the largest trivial (in the Tannakian sense) subobject. We can therefore apply the \textit{theorem of the fixed part} (Theorem 6.5.1 in \cite{FJ}), which is the following statement:
\begin{quote}
Let $M$ be an exponential motive with perverse realisation $A$, and denote by $A_0 \subseteq A$ the largest trivial subobject of $A$. There exists a classical submotive $M_0$ of $M$ such that the image of the perverse realisation of $M_0$ in $A$ is equal to $A_0$.
\end{quote}
In the case at hand, we deduce that $M$ contains a classical one-dimensional submotive $M_0$ such that the composition $M_0 \hookrightarrow M \to \Q (0)$ induces the identity on perverse realisations. Hence, $M_0$ is isomorphic to $\Q (0)$ and we obtain a section of the exact sequence. 
\end{proof}

\end{document}